\documentclass{amsart}
\usepackage{amsmath} 
\usepackage{amsthm}
\usepackage{amssymb}
\usepackage{mathrsfs}
\usepackage{hyperref}
\usepackage{color}
\usepackage{todonotes}

\newtheorem{theorem}{Theorem} [section]
\newtheorem*{remark}{Remark}
\newtheorem*{claim}{Claim}
\newtheorem{lemma}[theorem]{Lemma}
\newtheorem{proposition}[theorem]{Proposition}

\newtheorem{definition}[theorem]{Definition}

\numberwithin{equation}{section}

\begin{document}

\title{A  VARIATIONAL PRINCIPLE FOR NONLINEAR LOCAL PRESSURE }

\author{Jiayi Zhu}
\address{School of Mathematics and Statistics, Nanjing University of Information Science and Technology, Nanjing 210044,  PR China}
\email{zu112302@163.com }

\author{Rui Zou*}
\address{School of Mathematics and Statistics, Nanjing University of Information Science and Technology, Nanjing 210044,  PR China}
\email{zourui@nuist.edu.cn }

\thanks{* Rui Zou  is the corresponding author.  This work is  partially supported by  National Key R\&D Program of China (2022YFA1007800),  NSFC (12471185, 12271386).}

\date{\today}

\begin{abstract}
In this paper, we introduce  a  concept of nonlinear local  topological pressure defined via open covers and establish a corresponding variational principle. Furthermore, we provide multiple equivalent characterizations of  nonlinear pressure using different cover-based approaches.
\end{abstract}


\maketitle

\section{Introduction and main results}
\subsection{Nonlinear topological  pressure}
Topological pressure is a fundamental invariant in dynamical systems theory,  serving as a natural generalization of topological entropy. The concept was first introduced  by Ruelle \cite{R1} in the context of expansive systems and later extended by  Walters \cite{W1} to more general settings.  A \emph{topological dynamical system} (TDS for short) is a pair $(X,T)$ consisting of a compact metric space $X$ and  a surjective continuous  map $T:X \rightarrow X $. 
Let $\mathcal{M}(X)$ denote  the set of Borel probability measures on $X$, $\mathcal{M}(X,T)$ be the set of $T\text{-}$invariant Borel probability measures on $X$ and  $\mathcal{M}^e(X,T)$ be the set of ergodic  $T$-invariant measures on $X$. 
 Given a continuous function $f:X\to \mathbb{R}$,   the topological pressure $P(T, f)$  satisfies the following variational principle:
$$P(T, f)=\mathop{\sup}\limits_{\mu\in \mathcal{M}(X,T)}\left\{h_{\mu}(T)+\int_{X} f(x)d\mu (x) \right\},$$
where $h_{\mu}(T)$ is the measure-theoretic entropy of $\mu$. Topological pressure and its variational principle form the foundation of thermodynamic formalism and play a fundamental role in the dimension theory of dynamical systems.

Recently, Buzzi, Kloeckner and Leplaideur \cite{Buzzi}  developed a nonlinear thermodynamic formalism based on the Curie-Weiss mean-field theory \cite{Leplaideur19}. By transforming the statistical mechanics of generalized mean-field models into dynamical systems theory, they established a variational principle for nonlinear topological pressures, provided that the system possesses an abundance of ergodic measures. 
Subsequently, Barreira and Holanda \cite{Barreira2,Barreira1} extended this framework by introducing a higher-dimensional generalization of the nonlinear thermodynamic formalism and its continuous-time counterpart for flows, respectively.  Kucherenko \cite{Kucherenko} established a connection between the nonlinear thermodynamic formalism  and the theory of generalized rotation sets.  Yang, Chen and Zhou  \cite{Yangchenzhou} introduced  the notion of nonlinear weighted topological pressure for factor maps and established an associated variational principle.  Ding and Wang \cite{DW} introduced  the nonlinear topological pressure for   subsets and established corresponding variational principles.  

Now we recall the background  and the main result of  \cite{Buzzi}.  We call a function $\mathcal{E}\colon \mathcal{M}(X)\to\mathbb{R}$ is an \emph{energy} if  it is continuous  in the weak$\text{-}$star topology. For instance, given   continuous functions $f:X\to \mathbb{R}$ and  $F:\mathbb{R}\to \mathbb{R}$,  the function $\mathcal{E}$ defined  by 
$$\mathcal{E}(\mu)=F\left(\int fd \mu\right)$$
is then an energy on $\mathcal{M}(X)$.  For a given TDS $(X,T)$ and an energy $\mathcal{E}\colon \mathcal{M}(X)\to\mathbb{R}$, the {\em nonlinear topological pressure }$P(T,\mathcal{E})$ is defined as 
\[P(T,\mathcal{E})=\lim\limits _{\epsilon\rightarrow0}\limsup \limits _{n\rightarrow \infty}\frac{1}{n}\log P_{n}(T,\mathcal{E},\varepsilon),\]
where 
	$$P_{n}(T,\mathcal{E},\varepsilon)=\sup \left\{\mathop{\sum}\limits _{x\in E} e^{n\mathcal{E}(\triangle_{x}^{n})}:E \;is\; an\;(n,\varepsilon)\text{-} separated\;  set \;of\; X \right\},$$
	and $\triangle_{x}^{n}:=\frac{1}{n}\sum_{i=0}^{n-1}\delta_{T^{i}x}.$  
 Assuming that $(T,\mathcal{E})$ has an abundance of ergodic measures, they proved that  
$$P(T,\mathcal{E})=\mathop{\sup}\limits_{\mu\in \mathcal{M}(X,T)}\left\{h_{\mu}(T)+\mathcal{E}(\mu)\right\}.$$

%

\subsection{Local entropy and pressure} The local theory of entropy and pressure is a foundational framework in dynamical systems, with profound connections to various   concepts including entropy pairs, entropy sets, entropy points, and entropy structures, etc.  Given a TDS $(X,T)$ and an open cover $\mathcal{U}$ of $X,$
Romagnoli \cite{Romagnoli} introduced two types of measure-theoretic entropy relative to  $\mathcal{U}$: $h_{\mu}(T,\mathcal{U})$ and $h^{+}_{\mu}(T,\mathcal{U})$ (see Section 2 for precise  definitions). He established the following  variational principle for local entropy:
$$h_{top}(T,\mathcal{U})=\mathop{\sup}\limits_{\mu \in \mathcal{M}(X,T)}h_{\mu}(T,\mathcal{U}).$$
For invertible systems $(X,T)$, Glasner and Weiss \cite{Glasner05} proved that 
$$h_{top}(T,\mathcal{U})=\mathop{\sup}\limits_{\mu \in \mathcal{M}(X,T)}h_{\mu}^+(T,\mathcal{U}).$$
Further developments in local thermodynamic formalism were made by Huang, Ye, and Zhang \cite{Huangwenthree}, who derived a relative local variational principle for entropy.
  Building on these results Huang and Yi \cite{Huangwen} extended the  variational principle for local entropy to the case of pressure. Specifically, for a continuous potential function $f:X\to \mathbb{R}$, They showed the local pressure $P(T,f;\mathcal{U})$ satisfies
$$P(T, f; \mathcal{U})=\mathop{\sup}\limits_{\mu \in \mathcal{M}(X,T)}\left\{h_{\mu}(T,\mathcal{U})+\int_{X} f(x)d\mu (x)\right\}.$$
 In recent years, Wu \cite{WW}  investigated several notions of local pressure for subsets and measures, defined using the Carathéodory-Pesin construction. Cai \cite{CC} established the  variational principle for weighted local pressure. 
The study of local conditional pressures for  additive or sub-additive potentials was developed  by several authors (see, e.g., \cite{Romagnoli24,SongLi,Zhang09}). In the relative setting, local topological pressures were investigated in \cite{MaChenZhang,MaChen,Danilenko}. The case of amenable group actions   was studied in \cite{LiangYan,DooleyZhang,WangWuXiao}.

\subsection{Main result}
In this paper, we introduce the notion of nonlinear local topological pressure and prove  a corresponding variational principle.

Let $(X,T)$ be a {\rm TDS}. We define a \emph{cover} of $X$ as a finite collection of Borel subsets whose union equals X, and  a \emph{partition} of $X$ as a cover with pairwise disjoint elements.
Denote by $\mathcal{C}_{X}, \mathcal{P}_{X}$ and $\mathcal{C}_{X}^{o}$  the set of  \emph{covers}, \emph{partitions} and \emph{open covers} of $X$, respectively. Given two covers $\mathcal{U}, \mathcal{V} \in \mathcal{C}_{X}$, if each element of $\mathcal{U}$ is contained in some element of $\mathcal{V},$ then we say $\mathcal{U}$ is $finer$ than $\mathcal{V}$ (denoted by $\mathcal{U}\succeq\mathcal{V}$). Denote $\mathcal{U} \vee \mathcal{V}=\{{U}\cap V:U\in \mathcal{U}, V\in \mathcal{V}\}$ and  $\mathcal{U}^{n-1}_{0}=\bigvee ^{n-1}_{i=0}T^{-i}\mathcal{U}$.

 Recall that  $\mathcal{E}\colon \mathcal{M}(X)\to\mathbb{R}$ is called an energy if it is a continuous function on $\mathcal{M}(X)$.
\begin{definition}\label{ddd2.2} Let $(X,T)$ be a {\rm TDS}, $\mathcal{E}\colon \mathcal{M}(X)\to\mathbb{R}$ be an energy  and $\mathcal{U}\in \mathcal{C}_{X}^{o}.$ Let
	\begin{equation}\label{2.1}
		\begin{aligned}
			p_{n}(T,\mathcal{E};\mathcal{U})=\inf \left\{\mathop{\sum}\limits _{V\in \mathcal{V}}\mathop{\sup} \limits _{x\in V}e^{n\mathcal{E}(\triangle_{x}^{n})}:\mathcal{V}\in \mathcal{C}_{X} \; and \; \mathcal{V}\succeq\mathcal{U}_{0}^{n-1}\right\},
		\end{aligned}
	\end{equation}
	where $\triangle_{x}^{n}:=\frac{1}{n}\mathop{\sum}\limits _{i=0}^{n-1}\delta_{T^{i}x}.$ We define the  \emph{nonlinear local pressure of  $(T,\mathcal{E};\mathcal{U})$} by:
	\begin{equation}\label{local pressure}
		P(T, \mathcal{E}; \mathcal{U})=\limsup_{n\rightarrow\infty}\frac{1}{n}\log p_{n}(T,\mathcal{E};\mathcal{U}),
	\end{equation}
	and  define the  \emph{nonlinear lower local pressure of  $(T,\mathcal{E};\mathcal{U})$} by:
	\begin{equation}\label{lower local}
		\underline{P}(T, \mathcal{E}; \mathcal{U})=\liminf_{n\rightarrow\infty}\frac{1}{n}\log p_{n}(T,\mathcal{E};\mathcal{U}).
	\end{equation}
\end{definition}
Note that $\log p_{n}(T,\mathcal{E};\mathcal{U})$ is generally not sub-additive. As a result, the sequence $\{\frac{1}{n}\log p_{n}(T,\mathcal{E};\mathcal{U})\}_{n\in\mathbb{N}}$  may fail to converge. However, under the following abundance of ergodic measures condition,  we show  in  Theorem \ref{T2} that the limit $ \lim\limits_{n\rightarrow\infty}\frac{1}{n}\log p_{n}(T,\mathcal{E};\mathcal{U})$ exists.

\begin{definition}\label{D2.1}
We   say that {\em $(T,\mathcal{E};\mathcal{U})$ has  an abundance of ergodic measures}, if for any $\mu \in \mathcal{M}(X,T)$  and $\varepsilon>0,$ there is an ergodic measure $\nu \in \mathcal{M}^e(X,T)$ such that $h_{\nu}(T,\mathcal{U})+\mathcal{E}(\nu)>h_{\mu}(T,\mathcal{U})+\mathcal{E}(\mu)-\varepsilon.$
    
\end{definition}

Our main conclusions are as follows:

\begin{theorem}\label{T2}{\rm (Local variational principle):}
 Let $(X,T)$ be a {\rm TDS}, $\mathcal{E}\colon \mathcal{M}(X)\to\mathbb{R}$ be an energy and $\mathcal{U}\in \mathcal{C}_{X}^{o}$.  Suppose that $(T,\mathcal{E};\mathcal{U})$ has an abundance of ergodic measures. Then  the nonlinear local pressure $P(T,\mathcal{E};\mathcal{U})$   satisfies 
$$P(T,\mathcal{E};\mathcal{U})=\underline{P}(T, \mathcal{E}; \mathcal{U})=\mathop{\sup}\limits_{\mu \in \mathcal{M}(X,T)}\left\{h_{\mu}(T,\mathcal{U})+\mathcal{E}(\mu)\right\}.$$
\end{theorem}

In addition to the variational principle, we also demonstrate how the nonlinear local pressure relates to the global nonlinear topological pressure:
\begin{theorem}\label{T1}
 Let $(X,T)$ be a {\rm TDS}, $\mathcal{E}\colon \mathcal{M}(X)\to\mathbb{R}$ be an energy. Then 
\[P(T,\mathcal{E})=\sup \limits _{\mathcal{U}\in \mathcal{C}_{X}^{o}}P(T,\mathcal{E};\mathcal{U}),\text{~and \quad}	\underline{P}(T,\mathcal{E})=\sup \limits _{\mathcal{U}\in \mathcal{C}_{X}^{o}}\underline{P}(T,\mathcal{E};\mathcal{U}).\]
\end{theorem}
Theorem \ref{T1} is  a special case of Theorems \ref{Thm4.7} and \ref{Thm4.8}. In Section 4, we introduce several different definitions of nonlinear local pressure and   reveal their relationship with  nonlinear topological pressure.

The paper is organized as follows.  Section 2 introduces the fundamental definitions and  properties of local entropy and  pressure. In Section 3, we prove Theorem \ref{T2}, while Section 4 is devoted to the proof of Theorems \ref{Thm4.7} and \ref{Thm4.8}, from which Theorem \ref{T1} follows as a direct consequence.

\section{Preliminaries}

\subsection{Wasserstein distance on $\mathcal{M}(X)$ }
We equip the space of probability measures $\mathcal{M}(X)$ with the Wasserstein distance \cite{Villani}. For $\mu_{1}, \mu_{2}\in \mathcal{M}(X)$, the  distance is defined as:
$$W(\mu_{1},\mu_{2})=\sup \{\mu_{1}(f)-\mu_{2}(f):f \text{~ 1-Lipschitz function~}X\rightarrow\mathbb{R} \}.$$
Since  $X$ is compact, the Wasserstein distance induces the weak$\text{-}$star topology on $\mathcal{M}(X),$ and the Wasserstein distance can be bounded by the total variation distance \cite{Buzzi}:
$$W(\mu_{1},\mu_{2})\leq {\rm diam}(X) \|\mu_{1}-\mu_{2}\|_{TV}.$$

%
%
%
%
%
%
%
%

\subsection{Measure$\text{-}$theoretic entropies.}

Let $(X,T)$ be a TDS.
For $\mathcal{U}\in \mathcal{C}_{X}$, denote $H(\mathcal{U})=\log N(\mathcal{U}),$ where $N(\mathcal{U})$ is the minimum among the cardinals of the
sub-covers of $\mathcal{U}$. Then we define 
$$h_{top}(T,\mathcal{U})=\mathop{\lim}\limits_{n\rightarrow\infty}\frac{1}{n}H(\mathcal{U}_{0}^{n-1})=\mathop{\inf}\limits_{n\geq 1}\frac{1}{n}H(\mathcal{U}_{0}^{n-1}),$$ as the \emph{topological entropy of $\mathcal{U}$}.  Define
$$h_{top}(T)=\mathop{\sup}\limits_{\mathcal{U}\in \mathcal{C}_{X}^{o}}h_{top}(T,\mathcal{U})$$ as the \emph{topological entropy of $(X,T).$}

Given a partition $\alpha\in \mathcal{P}_{X}$ and a measure $\mu \in \mathcal{M}(X,T)$,  
 the {\em measure$\text{-}$theoretic entropy of $\mu$ relative to $\alpha$} is defined as : $$h_{\mu}(T,\alpha)=\mathop{\lim}\limits _{n\rightarrow\infty}\frac{1}{n}H_{\mu}(\alpha_{0}^{n-1})=\mathop{\inf}\limits _{n\geq 1}\frac{1}{n}H_{\mu}(\alpha_{0}^{n-1}),$$  
 where $$H_{\mu}(\alpha)=\mathop{\sum}\limits _{A\in \alpha}-\mu (A)\log \mu (A).$$
 the {\em measure$\text{-}$theoretic entropy of $\mu$} is then defined by: $$h_{\mu}(T)=\mathop{\sup}\limits _{\alpha \in \mathcal{P}_{X}}h_{\mu}(T,\alpha).$$ 

For a cover $\mathcal{U}\in \mathcal{C}_{X},$     Romagnoli  \cite{Romagnoli}  introduced the following two   {\em measure-theoretic entropies   relative to $\mathcal{U}$}: $$h_{\mu}(T,\mathcal{U})=\mathop{\lim}\limits _{n\rightarrow \infty}\frac{1}{n}H_{\mu}(\mathcal{U}_{0}^{n-1}) \text{~~and~~} h_{\mu}^{+}(T,\mathcal{U})=\mathop{\inf}\limits _{\alpha \succeq \mathcal{U}, \alpha\in \mathcal{P}_{X}}h_{\mu}(T,\alpha),$$ where $$ H_{\mu}(\mathcal{U})=\mathop{\inf}\limits _{\alpha \succeq \mathcal{U}, \alpha\in \mathcal{P}_{X}}H_{\mu}(\alpha).$$

By \cite[Lemma 8]{Romagnoli}, for  $\mu \in \mathcal{M}(X)\; and\; \mathcal{U} , \mathcal{V} \in \mathcal{C}_{X},$  $ H_{\mu}(\mathcal{U})$ satisfies the following properties:
\begin{enumerate}
\item[(i)]0$\leq H_{\mu}(\mathcal{U})\leq \log N(\mathcal{U}).$ 
\item[(ii)] If $\mathcal{U} \succeq \mathcal{V}$, then $H_{\mu}(\mathcal{U})\geq H_{\mu}(\mathcal{V}).$
\item[(iii)] $H_{\mu}(\mathcal{U}\vee \mathcal{V})\leq H_{\mu}(\mathcal{U})+ H_{\mu}(\mathcal{V}).$
\item[(iv)] $H_{\mu}(T^{-1}\mathcal{U})\leq H_{T_{*}\mu}(\mathcal{U})$, where $T_{*}\mu =\mu \circ T^{-1}$, and, the equality holds when $(X,T)$ is invertible.
\end{enumerate}
 Moreover, if   $\mu \in \mathcal{M}(X,T),$ it's proved  by \cite[Lemma 2.7]{Huangwen}  that
\begin{equation}\nonumber
	\begin{aligned}
		h_{\mu}(T)=\sup _{\mathcal{U} \in \mathcal{C}_{X}^{o}}h_{\mu}(T,\mathcal{U}).
	\end{aligned}
\end{equation} 

Now given  $\mathcal{U}=\{U_{1}, U_{2},..., U_{d}\}\in \mathcal{C}_X^o,$  denote
 $$\mathcal{U}^{*}=\{\alpha \in \mathcal{P}_{X}: \alpha =\{A_{1}, A_{2},...,A_{d}\}, A_{i}\subset U_{i}, i=1,2,...,d\},$$
 where  $A_i$ could be empty for some $i.$   By \cite[Lemma 2]{Huangwentwo} and \cite[Corollary]{Huangwen}, the following properties hold.
 \begin{proposition}\label{prop h^+}
 		Let $(X,T)$ be a {\rm TDS}, $\mu\in \mathcal{M}(X,T)$ and $\mathcal{U}\in \mathcal{C}_X^o$.  The following properties hold:
 		\begin{enumerate}
 			\item $h_{\mu}^{+}(T,\mathcal{U})=\mathop{\inf}\limits_{\alpha\in\mathcal{U}^{*}}h_{\mu}(T,\alpha);$
 			\item If in addition $(X,T)$ is invertible, then  $h_{\mu}(T,\mathcal{U})=h_{\mu}^{+}(T,\mathcal{U}).$
 		\end{enumerate}
 	 
 \end{proposition}

\subsection{Some properties of nonlinear local pressures}
Let $(X,T)$ be a {\rm TDS}, $\mathcal{E}\colon \mathcal{M}(X)\to\mathbb{R}$ be an energy and $\mathcal{U}\in \mathcal{C}_{X}^{o}$.   The nonlinear local pressure 	$P(T, \mathcal{E}; \mathcal{U})=\limsup_{n\rightarrow\infty}\frac{1}{n}\log p_{n}(T,\mathcal{E};\mathcal{U})$  is defined as in Definition \ref{ddd2.2}. We first  provide an equivalent statement of $p_{n}(T,\mathcal{E};\mathcal{U})$.

 For $\mathcal{V}\in\mathcal{C}_{X}$, let $\alpha$ be the Borel partition generated by $\mathcal{V}.$ Define $$\mathcal{P}^{*}(\mathcal{V})=\{\beta \in \mathcal{P}_{X}:\beta\succeq\mathcal{V} \text{ and each atom  of } \beta \text{ is the union  of  some  atoms of  }\alpha\}.$$

\begin{lemma}\label{L4.1}
	$\mathcal{P}^{*}(\mathcal{V})$ is a finite set, and for each $n \in\mathbb{N}$,
	$$\mathop {\inf}\limits_ {\beta \in \mathcal{C}_{X}, \beta\succeq\mathcal{V}} \mathop{\sum}\limits_ {B\in \beta} \mathop{\sup}\limits_ {x\in B} e^{n \mathcal{E}(\triangle_{x}^{n})}=\min\left\{\mathop{\sum}\limits_ {B\in \beta} \mathop{\sup}\limits_ {x\in B} e^{n \mathcal{E}(\triangle_{x}^{n})}:\beta \in P^{*}(\mathcal{V})\right\}.$$ 
	In particular, by taking $\mathcal{V}=\mathcal{U}_{0}^{n-1}$, we have  
	$$p_{n}(T,\mathcal{E};\mathcal{U})=\min \left\{\mathop{\sum}\limits_ {B\in \beta}\mathop{\sup}\limits_ {x\in B} e^{n \mathcal{E}(\triangle_{x}^{n})}:\beta \in \mathcal{P}^{*}(\mathcal{U}_{0}^{n-1}) \right\}.$$
\end{lemma}

\begin{proof}
	The proof of \cite[Lemma 2.1]{Huangwen}  remains valid for  this lemma  upon replacing $f_n(x)$ by $n \mathcal{E}(\triangle_{x}^{n}).$  We therefore omit the details of the proof.
\end{proof}

Let  $(Y,S)$ and $(X,T)$ be two TDS.  We say $\pi:Y\rightarrow X$   is   a {\em factor map}   from $(Y,S)$ to $(X,T)$ if   $\pi$ is a surjective and continuous map   satisfying $\pi\circ S=T\circ\pi.$  In this case, $(Y,S)$ is called an \emph{extension} of $(X,T)$, and  $(X,T)$ is called a \emph{factor} of $(Y,S).$ In addition, if $\pi $ is also   injective then it's called   an \emph{isomorphism}.  A map  $\pi:Y\rightarrow X$ induces a map $\pi_*:\mathcal{M}(Y)\rightarrow \mathcal{M}(X)$  defined by $\pi_*\mu=\mu\circ\pi^{-1}$.

\begin{lemma}[\cite{Romagnoli}, Proposition 6]\label{lem 2.5}
	Let  $(Y,S)$ and $(X,T)$ be two {\rm TDS}.	If $\pi :(Y,S)\rightarrow (X,T)$ is a factor map, then for any $\nu \in \mathcal{M}_S(Y)$ and $\mathcal{U}\in \mathcal{C}_{X}^{o}$, we have $h_{\nu}(S,\pi ^{-1}\mathcal{U})=h_{\pi_* \nu }(T,\mathcal{U}).$      
\end{lemma}

\begin{lemma}\label{lem 2.4}
	Let $\pi :(Y,S)\rightarrow (X,T)$ ba a factor map,  $\mathcal{E}\colon \mathcal{M}(X)\to\mathbb{R}$ be an energy and $\mathcal{U}\in \mathcal{C}_{X}^{o}$.  Then $P(S,\mathcal{E}\circ \pi_* ;\pi ^{-1}\mathcal{U})=P(T,\mathcal{E};\mathcal{U}).$
\end{lemma}

\begin{proof}
	Fix $n\in \mathbb{N}$ and let $\mathcal{V} \in \mathcal{C}_{X}$ satisfy $\mathcal{V}\succeq \mathcal{U}_{0}^{n-1}.$ Then  $\pi ^{-1}\mathcal{V}\in \mathcal{C}_{Y}$ and $\pi ^{-1}\mathcal{V} \succeq(\pi ^{-1}\mathcal{U})_{0}^{n-1}.$   It follows from $\pi\circ S=T\circ\pi$ that
	\begin{align*}
		\mathop{\sum}\limits _{V\in \mathcal{V}}\sup_{x\in V}e^{n\mathcal{E}(\triangle_{x}^{n})}=\mathop{\sum}\limits _{V\in \mathcal{V}}\sup_{y\in \pi ^{-1}(V)}e^{n\mathcal{E}(\triangle_{\pi y}^{n})}
		&=\mathop{\sum}\limits _{V\in \mathcal{V}}\sup_{y\in \pi ^{-1}(V)}e^{n\mathcal{E}\circ \pi_*(\triangle_{y}^{n})}\\
		&\geq p_{n}(S,\mathcal{E}\circ \pi_* ; \pi ^{-1}\mathcal{U}).
	\end{align*}
	Since $\mathcal{V}$ is arbitrary, 
	$$p_{n}(T,\mathcal{E};\mathcal{U})\geq p_{n}(S,\mathcal{E}\circ \pi_*; \pi ^{-1}\mathcal{U}).$$
	
	On the other hand,  by Lemma \ref{L4.1},
	$$p_{n}(S,\mathcal{E}\circ \pi_*; \pi ^{-1}\mathcal{U})=\mathop{\min}\limits_{\beta \in \mathcal{P}^{*}\left((\pi ^{-1}\mathcal{U})_{0}^{n-1}\right)}\left\{ \mathop{\sum}\limits _{B\in \beta }\mathop{\sup}\limits_{y\in B }e^{n\mathcal{E}\circ \pi_*(\triangle_{y}^{n})}\right\}.$$
	Note that  if $\beta =\{ B_{1},B_{2},\dots ,B_{m}\}\in \mathcal{P}^{*}\left((\pi ^{-1}\mathcal{U})_{0}^{n-1}\right),$ then 
	$$\pi \beta =\{ \pi (B_{1}),\pi (B_{2}),\dots ,\pi (B_{m})\}\in \mathcal{C}_{X} \text{~and~}\pi \beta \succeq \mathcal{U}_{0}^{n-1}.$$
	Hence 
	$$\mathop{\sum}\limits _{i=1}^{m}\sup_{y\in B_{i}}e^{n\mathcal{E}\circ \pi_*(\triangle_{y}^{n})}=\mathop{\sum}\limits _{i=1}^{m}\sup_{x\in \pi (B_{i})}e^{n\mathcal{E}(\triangle_{x}^{n})}\geq p_{n}(T,\mathcal{E};\mathcal{U}).$$ 
	By the arbitrariness  of $\beta,$ we have $p_{n}(S,\mathcal{E}\circ \pi_*; \pi ^{-1}\mathcal{U})\ge p_{n}(T,\mathcal{E};\mathcal{U}).$
	
	Therefore,  $p_{n}(S,\mathcal{E}\circ \pi_*; \pi ^{-1}\mathcal{U})= p_{n}(T,\mathcal{E};\mathcal{U})$,  and the conclusion of the lemma follows.
\end{proof}

\section{ Proof of Theorem \ref{T2}.} 

 In this section,  we prove Theorem \ref{T2}.
Let $(X,T)$ be a {\rm TDS}, $\mathcal{E}\colon \mathcal{M}(X)\to\mathbb{R}$ be an energy and $\mathcal{U}\in \mathcal{C}_{X}^{o}$.  We prove Theorem \ref{T2} by showing
 
\begin{equation}\label{3.1}
\begin{split}
	\mathop{\sup}\limits _{\mu \in \mathcal{M}^e(X,T)}\{h_{\mu}(T,\mathcal{U})+\mathcal{E}(\mu)\}\mathop{\leq}\limits_{\textcircled{1}}   \underline{P}(T,\mathcal{E};\mathcal{U})&\mathop{\leq}\limits_{\textcircled{2}} {P}(T,\mathcal{E};\mathcal{U})\\
	&\mathop{\leq}\limits_{\textcircled{3}} \mathop{\sup}\limits _{\mu \in \mathcal{M}(X,T)}\{h_{\mu}(T,\mathcal{U})+\mathcal{E}(\mu)\}\\
	&\mathop{\leq}\limits_{\textcircled{4}} \mathop{\sup}\limits_ {\mu \in \mathcal{M}^e(X,T)} \{h_{\mu}(T, \mathcal{U}) + \mathcal{E}(\mu)\}
\end{split}
\end{equation}

Inequality $\textcircled{1}$ is proved in Proposition \ref{P4.4}.  Inequality $\textcircled{2}$  follows from the definitions of $\underline{P}(T,\mathcal{E};\mathcal{U})$ and ${P}(T,\mathcal{E};\mathcal{U}).$ Inequality $\textcircled{4}$ is proved in Proposition \ref{Prop 3.3} under the assumption that $(T,\mathcal{E};\mathcal{U})$ has an abundance of ergodic measures. Finally,  Inequality $\textcircled{3}$ is proved in Subsection 3.2, Proposition \ref{P4.11}. 

\subsection{Proof of inequality $\textcircled{1}$ and $\textcircled{4}$ in  \eqref{3.1}.} We begin by recalling the following basic lemma.

\begin{lemma}[\cite{P.W}, Lemma 9.9]\label{L4.2}
Let $a_{1} , a_{2} ,\dots, a_{k}$ be given real numbers. If $b_{i}\geq 0$ for every $1\le i\le k,$ and $\sum^{k}_{i=1}b_{i}=1$, then $$\mathop{\sum}\limits_ {i=1}\limits^{k}b_{i}(a_{i}-\log b_{i})\leq \log(\mathop{\sum}\limits_ {i=1}\limits^{k}e^{a_{i}}),$$ and equality holds iff $b_{i}=\frac{e^{a_{i}}}{\sum^{k}_{i=1}e^{a_{i}}}$ for all $i= 1, 2,\dots, k.$
\end{lemma}

\begin{proposition}\label{P4.4}$\mathrm{(Inequality \textcircled{1})}.$
Let $(X,T)$ be a {\rm TDS},  $\mathcal{E}\colon \mathcal{M}(X)\to\mathbb{R}$ be an energy and $\mathcal{U}\in \mathcal{C}_{X}^{o}$.  Then 
$$\sup\limits_{\mu\in \mathcal{M}^e(X,T)}\left\{h_{\mu}(T, \mathcal{U}) + \mathcal{E}(\mu)\right\} \leq \underline{P}(T, \mathcal{E}; \mathcal{U}).$$
\end{proposition}

\begin{proof}
Fix any $\mu \in \mathcal{M}^e(X,T).$ Since $\mathcal{E}$ is uniformly continuous on $\mathcal{M}(X)$, for any $\epsilon>0,$ there exists $
\delta>0$ such that for any $\nu\in \mathcal{M}(X)$ with $W(\nu,\mu)< \delta,$  one has  $|\mathcal{E}(\nu)-\mathcal{E}(\mu)|<\epsilon.$ Since  $\mu \in \mathcal{M}^e(X,T),$  for $\mu$-almost all $x\in X$,  $\triangle_{x}^{n} \rightarrow\mu$ as $n\to \infty$.
Then by Egorov’s Theorem, there exists $N\geq 1$ and a measurable subset $\Gamma$ with $\mu(\Gamma)>1-\epsilon$ such that $W(\triangle_{x}^{n},\mu)<\delta$ for every $x\in \Gamma$ and  $n\ge N$. It follows that 
\begin{equation}\label{3.5.5}
	 |\mathcal{E}(\triangle_{x}^{n})-\mathcal{E}(\mu)|<\epsilon,\quad x\in \Gamma, n\ge N. 
\end{equation}
By Lemma \ref{L4.1},  for any $n \in \mathbb{N}$, there exists a finite partition $\beta \in \mathcal{P}^{*}(\mathcal{U}_{0}^{n-1})$ such that $$\mathop{\sum}\limits_ {B\in \beta} \mathop{\sup}\limits_ {x\in B} e^{n \mathcal{E}(\triangle_{x}^{n})}=p_{n}(T,\mathcal{E};\mathcal{U}).$$ It follows from Lemma \ref{L4.2} that 
\begin{equation}\label{3.3}
\begin{aligned}
\log p_{n}(T,\mathcal{E};\mathcal{U})&=\log \mathop{\sum}\limits_ {B\in \beta} \mathop{\sup}\limits_ {x\in B} e^{n \mathcal{E}(\triangle_{x}^{n})}\\
&\geq \sum_{B \in \beta}\mu (B)\left(\mathop{\sup_{x \in B}n \mathcal{E}(\triangle_{x}^{n})}-\log\mu(B)\right)\\&=H_{\mu}(\beta)+\mathop{\sum}\limits_ {B\in \beta} \mathop{\sup}\limits_ {x\in B}n \mathcal{E}(\triangle_{x}^{n})\cdot\mu(B).
\end{aligned}
\end{equation}

We  focus on the latter item. Let $\beta _{1}=\{B\in \beta : B \cap  \Gamma\neq \varnothing\}.$ Then 
\[\sum_{B\in \beta_{1}} \mu(B) \geq \sum_{B\in \beta_{1}} \mu(B \cap \Gamma) = \mu \left( \bigcup_{B\in\beta_{1}} (B \cap \Gamma) \right) = \mu(\Gamma) > 1 - \epsilon.\]
Therefore,  \eqref{3.5.5} and \eqref{3.3}  give  
\begin{equation}\nonumber
\begin{aligned}
\log (p_{n}(T, \mathcal{E}; \mathcal{U}))&\geq H_{\mu}(\beta)+\mathop{\sum}\limits_ {B\in \beta_{1}} \mathop{\sup}\limits_ {x\in B}n \mathcal{E}(\triangle_{x}^{n})\cdot\mu(B)\\&\geq H_{\mu}(\beta)+\mathop{\sum}\limits_ {B\in \beta_{1}} n( \mathcal{E}(\mu)-\epsilon)\cdot\mu(B)\\&\geq H_{\mu}(\beta)+ n( \mathcal{E}(\mu)-\epsilon)(1-\epsilon)\\
&\geq H_{\mu}(\mathcal{U}_{0}^{n-1})+ n( \mathcal{E}(\mu)-\epsilon)(1-\epsilon),
\end{aligned}
\end{equation}
where the last inequality follows from  $\beta\succeq \mathcal{U}_{0}^{n-1}$. Dividing the above equation by $n$ and taking the liminf  as $n\rightarrow\infty$, one has 
$$ \underline{P}(T, \mathcal{E}; \mathcal{U})\ge h_{\mu}(T, \mathcal{U}) + ( \mathcal{E}(\mu)-\epsilon)(1-\epsilon).$$
 By the arbitrariness of $\epsilon$, the result follows.
\end{proof}

 \begin{proposition}\label{Prop 3.3}{\rm{(Inequality \textcircled{4})}.}
 	Let $(X,T)$ be a {\rm TDS}, $\mathcal{E}\colon \mathcal{M}(X)\to\mathbb{R}$ be an energy and $\mathcal{U}\in \mathcal{C}_{X}^{o}$. Suppose that $(T,\mathcal{E};\mathcal{U})$ has an abundance of ergodic measures. Then 
 	$$\mathop{\sup}\limits _{\mu \in \mathcal{M}(X,T)}\{h_{\mu}(T,\mathcal{U})+\mathcal{E}(\mu)\}\mathop{\leq} \mathop{\sup}\limits_ {\mu \in \mathcal{M}^e(X,T)} \{h_{\mu}(T, \mathcal{U}) + \mathcal{E}(\mu)\}.$$
 \end{proposition}
 
 \begin{proof}
 	By Definition \ref{D2.1},  for any $\mu \in \mathcal{M}(X,T)$ and any $\varepsilon>0$, there is an ergodic measure $\nu \in \mathcal{M}^e(X,T)$ such that $h_{\mu}(T,\mathcal{U})+\mathcal{E}(\mu) \leq h_{\nu}(T,\mathcal{U})+\mathcal{E}(\nu)+\varepsilon,$ which implies  
 	$$h_{\mu}(T,\mathcal{U})+\mathcal{E}(\mu) \leq \mathop{\sup}\limits_ {\nu \in \mathcal{M}^e(X,T)}\left\{h_{\nu}(T,\mathcal{U})+\mathcal{E}(\nu)\right\}+\varepsilon.$$
 	 Since $\mu \in \mathcal{M}(X,T)$ and   $\varepsilon>0$ are arbitrary, 
 	$$\mathop{\sup}\limits _{\mu \in \mathcal{M}(X,T)}\{h_{\mu}(T,\mathcal{U})+\mathcal{E}(\mu)\}\mathop{\leq} \mathop{\sup}\limits_ {\mu \in \mathcal{M}^e(X,T)} \{h_{\mu}(T, \mathcal{U}) + \mathcal{E}(\mu)\}.$$
 \end{proof}

\subsection{Proof of inequality $\textcircled{3}$ in  \eqref{3.1}.}
  To establish Inequality \textcircled{3}, the following lemma is crucial. 
 
\begin{lemma}\label{L4.5}
  Suppose that  $\left \{ \alpha _l\right \} _{l\ge 1}$ is a sequence of finite  partitions of $X$ which are finer than $\mathcal{U}$. Then for any $\gamma>0,$ there exist  $N=N(\gamma)\in\mathbb{N}, E\in \mathbb{R}$ and   a sequence of finite subsets  $\{D_{n_k}\}_{k\ge 1}$  of $X$   where $n_k\to\infty$, such that for any $k\ge 1,$
\begin{enumerate} 
	\item  for all $1\le l\le n_k$, each atom of $\left (\alpha _l\right) _{0}^{n_k-1}$ contains at most one point of $D_{n_k}$;
	\item $\sum_{x\in D_{n_k}} e^{n_k \mathcal{E}(\triangle_{x}^{n_k})} \geq \frac{1}{2n_kN}\cdot e^{n_k\left(P(T, \mathcal{E}; \mathcal{U})-\gamma\right) }$;
	\item 	For any $\mu \in \mathcal{M}(X)$ that is a convex combination of $\{\triangle_{x}^{n_k}\}_{x\in D_{n_k}}$, we have 
	$$| \mathcal{E}(\mu)-E| \leq \gamma.$$
\end{enumerate}
\end{lemma}

\begin{proof}
We begin by proving the following claim.
\begin{claim}
	For any $n\ge 1$, there exists a   finite subset  $B_{n}$  of $X$  such that for any    $1\le l\le n$, each atom of $\left (\alpha _l\right) _{0}^{n-1}$ contains at most one point of $B_{n}$.  Moreover, $$\sum_{x\in B_{n}} e^{n \mathcal{E}(\triangle_{x}^{n})} \geq \frac{1}{2n}\cdot p_{n}(T, \mathcal{E}; \mathcal{U}).$$
\end{claim}
\begin{proof}[Proof of the Claim]
Fix any $n\in \mathbb{N}$.	For each $x\in X,$ denote by $A_l(x)$ the  atom of $\left (\alpha _l\right) _{0}^{n-1}$ which contains $x.$ We now construct the set $B_n$.  Choose $x_1\in X$ such that $e^{n  \mathcal{E}(\triangle_{x_1}^{n})}=\max_{x\in X}e^{n  \mathcal{E}(\triangle_{x}^{n})}$ and denote $X_1=X$. By induction, suppose that $X_i$ and $x_i\in X_i$ have been defined.  Let $X_{i+1}=X_i\setminus\bigcup_{l=1}^n A_l(x_i),$ and choose $x_{i+1}\in X_{i+1}$ such that $e^{n  \mathcal{E}(\triangle_{x_{i+1}}^{n})}\ge\frac{1}{2}\sup_{x\in X_{i+1}}e^{n  \mathcal{E}(\triangle_{x}^{n})}$.  Since $\{A_l(x):1\le l\le n, x\in X\}$ is a finite set,  there exists $m\ge 1$ such that $X_m\setminus\bigcup_{l=1}^n A_l(x_m)=\varnothing.$ Then we obtain  a finite set $B_n:=\{x_1,\dots,x_m\}$ such that for each $1\le i\le m,$
\begin{itemize}
	\item[(1)]   $e^{n  \mathcal{E}(\triangle_{x_{i}}^{n})}\ge\frac{1}{2}\sup_{x\in X_{i}}e^{n  \mathcal{E}(\triangle_{x}^{n})};$ 
	\item[(2)] $X_i= X_{i+1}\cup \left(\bigcup_{l=1}^n A_l(x_i)\cap X_i   \right),$ where $X_{m+1}=\varnothing$.
\end{itemize}
By   construction,  for all $1\le l\le n$, each atom of $\left (\alpha _l\right) _{0}^{n-1}$ contains at most one point of $B_{n}$.
Moreover,  
\begin{align*}
	e^{n\mathcal{E}(\triangle_{x_{i}}^{n})}\ge\frac{1}{2}\sup_{x\in X_{i}}e^{n  \mathcal{E}(\triangle_{x}^{n})}\ge \frac{1}{2}\sup_{x\in A_l(x_i)\cap X_i}e^{n  \mathcal{E}(\triangle_{x}^{n})}, \quad \forall 1\le l \le n.
\end{align*}
Let $\beta=\{A_l(x_i)\cap X_i:1\le l\le n, 1\le i\le m\}.$ Then $\beta\in \mathcal{C}_X$ and $\beta\succeq \mathcal{U}_0^{n-1}$. Therefore, 
\begin{align*}
	\sum_{x\in B_{n}} e^{n \mathcal{E}(\triangle_{x}^{n})}=\sum_{i=1}^{m}e^{n\mathcal{E}(\triangle_{x_{i}}^{n})}
	&\ge \sum_{i=1}^{m}\frac{1}{n}\sum_{l=1}^{n}\frac{1}{2}\sup_{x\in A_l(x_i)\cap X_i}e^{n  \mathcal{E}(\triangle_{x}^{n})}\\
	&=\frac{1}{2n}\sum_{B\in\beta}\sup_{x\in B}e^{n  \mathcal{E}(\triangle_{x}^{n})}\\
	&\ge \frac{1}{2n}\cdot p_{n}(T, \mathcal{E}; \mathcal{U}).
\end{align*}
\end{proof}
We now continue the proof of Lemma \ref{L4.5}.  Since $\mathcal{E}$ is uniformly continuous on $\mathcal{M}(X)$, for any $\gamma>0,$ there exists $\delta>0$ such that  for all $\mu,\nu \in \mathcal{M}(X)$ satisfying $W(\mu ,\nu )\leq \delta,$ one has $|\mathcal{E}(\mu)-\mathcal{E}(\nu)|\leq \gamma.$

Let $S=\{\nu_{1},\nu_{2},\dots,\nu_{N}\}$ be a $\delta\text{-}$dense subset of    $\mathcal{M}(X)$. For  each $1\le i\le N,$ denote $B(\nu_i,\delta)=\{\mu \in \mathcal{M}(X):W(\mu, \nu_i )\leq \delta\}$. Set $V_1=B(\nu_1, \delta)$ and $V_i=B(\nu_i, \delta)\setminus\bigcup_{j=1}^{i-1}V_j$ for $2\le i\le N.$ Then  $\{V_{i}\}_{1\le i\le N}$ is a partition of $\mathcal{M}(X)$, and   $|\mathcal{E}(\mu)-\mathcal{E}(\nu_i)|\leq \gamma$ for all $\mu \in V_{i}.$ 

Let $\{n_k\}_{k\ge 1}$ be a sequence of integers diverging to infinity such that 
\begin{equation*} 
	\lim\limits_{k\to\infty} \frac{1}{n_k}\log p_{n_k}(T,\mathcal{E};\mathcal{U})=P(T, \mathcal{E}; \mathcal{U}).
\end{equation*}
 Then there exists $K_0\in \mathbb{N}$ such that  $p_{n_k}(T,\mathcal{E};\mathcal{U})\ge e^{n_k\left(P(T, \mathcal{E}; \mathcal{U})-\gamma\right)}$ holds for every $k\ge K_0.$
Denote ${C}_{n_k,i}=\{x\in B_{n_k}:\triangle_{x}^{n_k}\in V_{i}\}$. Then by the Claim above,  for every $k\ge K_0,$
\begin{align*}
	\sum_{x\in B_{n_k} }e^{n_k \mathcal{E}(\triangle_{x}^{n_k})} =\sum_{i=1}^N\sum_{x\in {C}_{n_k,i} }e^{n_k \mathcal{E}(\triangle_{x}^{n_k})}
	&\ge \frac{1}{2n_k}\cdot p_{n_k}(T, \mathcal{E}; \mathcal{U})\\
	&\ge \frac{1}{2n_k}\cdot e^{n_k\left(P(T, \mathcal{E}; \mathcal{U})-\gamma\right)}.
\end{align*}
It follows that  for each $k\ge K_0$  there exists $1\le i\le N$ such that 
$$\sum_{x\in {C}_{n_k,i} }e^{n_k \mathcal{E}(\triangle_{x}^{n_k})}\ge  \frac{1}{2n_kN}\cdot e^{n_k\left(P(T, \mathcal{E}; \mathcal{U})-\gamma\right)}.$$
 Therefore, there exists $i_0\in [1,N]$ and a subsequence of  $\{n_k\}_{k\ge 1}$(still denoted by $\{n_k\}_{k\ge 1}$) such that
 $$\sum_{x\in {C}_{n_k,i_0} }e^{n \mathcal{E}(\triangle_{x}^{n_k})}\ge   \frac{1}{2n_kN}\cdot e^{n_k\left(P(T, \mathcal{E}; \mathcal{U})-\gamma\right)}, \quad \forall k\ge 1.$$
 
 Denote $D_{n_k}={C}_{n_k,i_0}$ and $E=\mathcal{E}(\nu_{i_0}).$ Then $D_{n_k}$ satisfies (\romannumeral1) and  (\romannumeral2) of Lemma \ref{L4.5}. To check  (\romannumeral3),  suppose that a probability measure $\mu=\sum _{x\in {D}_{n_k}}\lambda_{x}\triangle_{x}^{n_{k}}$ is a convex combination of the measures $\triangle_{x}^{n_{k}}$, then  it follows from $W(\triangle_{x}^{n_{k}},\nu_{i_0})\le \delta$ for every $x\in D_{n_k}$ that $W(\mu,\nu_{i_0})\le \delta$. Hence $|\mathcal{E}(\mu)-E|\leq|\mathcal{E}(\mu)-\mathcal{E}(\nu_{i_0})|\leq \gamma.$
\end{proof}

Now we are ready to prove   inequality $\textcircled{3}$ in  \eqref{3.1}.
\begin{proposition}\label{P4.11}{\rm (Inequality $\textcircled{3}$).}
	Let $(X,T)$ be a {\rm TDS}, $\mathcal{E}\colon \mathcal{M}(X)\to\mathbb{R}$ be an energy and  $\mathcal{U}\in\mathcal{C}_X^0$. Then
	$$ P(T,\mathcal{E};\mathcal{U})\leq \mathop{\sup}\limits _{\mu \in \mathcal{M}(X,T)}\{h_{\mu}(T,\mathcal{U})+\mathcal{E}(\mu)\}.$$
\end{proposition}

\begin{proof}
	We split the proof into two steps.
	
	{\bfseries Step 1.} Suppose that $(X,T)$ is an   invertible, zero-dimensional TDS (i.e., it has a topological base consisting of clopen subsets).
Denote $\mathcal{U}=\{U_{1}, U_{2},..., U_{d}\}$ and let $$\mathcal{U}^{*}=\{\alpha \in \mathcal{P}_{X}: \alpha =\{A_{1}, A_{2},...,A_{d}\}, A_{i}\subset U_{i}, i=1,2,...,d\}.$$ 

 Since $X$ is zero-dimensional,   $\mathcal{U}^{*}$ forms a countable set of partitions. Moreover, these partitions are   finer than  $\mathcal{U}$  and  consist entirely of clopen sets.  Denote  $\mathcal{U}^{*}= \{\alpha _{l}:l\ge 1\}$.

Then by Lemma \ref{L4.5}, for any $\gamma>0,$ there exist  $N=N(\gamma)\in\mathbb{N}, E\in \mathbb{R}$ and   a sequence of finite subsets  $\{D_{n_k}\}_{k\ge 1}$  of $X$   where $n_k\to\infty$, such that for any $k\ge 1,$
\begin{enumerate} 
	\item  for all $1\le l\le n_k$, each atom of $\left (\alpha _l\right) _{0}^{n_k-1}$ contains at most one point of $D_{n_k}$;
	\item $\sum_{x\in D_{n_k}} e^{n_k \mathcal{E}(\triangle_{x}^{n_k})} \geq     \frac{1}{2n_kN}\cdot e^{n_k\left(P(T, \mathcal{E}; \mathcal{U})-\gamma\right)}$;
	\item if $\mu \in \mathcal{M}(X)$ is a convex combination of the measures $\triangle_{x}^{n_{k}}$ where $x$ runs over $D_{n_k}$, then $| \mathcal{E}(\mu)-E| \leq \gamma.$
\end{enumerate}
Let
$$\nu_{k}:=\sum_{x\in{D}_{n_k}}{\lambda_{n_k}(x)}\delta_{x},$$
where $\lambda_{n_k}(x)=\frac{e^{n_{k} \mathcal{E}(\triangle_{x}^{n_{k}})} }{\sum_{y\in{D}_{n_k}}e^{n_{k} \mathcal{E}(\triangle_{y}^{n_{k}})}}.$ 
Define the empirical measure
$$\mu_{k}:=\frac{1}{n_{k}}\sum_{i=0}^{n_{k}-1}T_{*}^{i}\nu_{k}=\sum_{x\in{D}_{n_k}}{\lambda_{n_k}(x)}\triangle_{x}^{n_{k}}.$$
Up to extracting a further subsequence, we may  assume $\mu _{\infty}=\lim\limits_{k\to\infty}\mu _{k}$.
Then $\mu _{\infty}\in \mathcal{M}(X,T)$.   Moreover, by $| \mathcal{E}(\mu_k)-E| \leq \gamma$ for every $k\ge 1$, we also obtain 
\begin{equation}\label{eq3.5}
	| \mathcal{E}(\mu_\infty)-E| \leq \gamma.
\end{equation}
 Since $T$ is invertible, by Proposition \ref{prop h^+}, $$h_{\mu_\infty}(T,\mathcal{U})=h_{\mu_\infty}^{+}(T,\mathcal{U})=\mathop{\inf}\limits_{\alpha\in\mathcal{U}^{*}}h_{\mu_\infty}(T,\alpha)=\mathop{\inf}\limits_{l\ge 1}h_{\mu_\infty}(T,\alpha_{l}).$$ 
It suffices to prove for each $l\geq 1$: 
\[{P}(T, \mathcal{E}; \mathcal{U})\le  h_{\mu_\infty}(T,\alpha_{l})+\mathcal{E}(\mu_\infty)+3\gamma.  \]

Since each atom of $\left (\alpha _l\right) _{0}^{n_k-1}$ contains at most one point of $D_{n_k}$, it follows that
 $$H_{\nu_{k}}\left((\alpha_{l})_{0}^{n_{k}-1}\right)= \sum_{x \in D_{n_k}}\lambda_{n_k}(x)\log \frac{1}{\lambda_{n_k}(x)}.$$
By Lemma \ref{L4.5} and inequality \eqref{eq3.5}, we have 
\begin{align*}
	 \log \frac{1}{\lambda_{n_k}(x)}&\ge \log\frac{1}{2Nn_k}+n_k\left(P(T, \mathcal{E}; \mathcal{U})-\gamma\right)- n_k \mathcal{E}(\triangle_{x}^{n_k})\\
	                                                         &\ge n_k\left(P(T, \mathcal{E}; \mathcal{U})-\gamma\right)- n_k\cdot(E+\gamma)-\log{(2Nn_k)}\\
	                                                         &\ge  n_k\left(P(T, \mathcal{E}; \mathcal{U})-\gamma\right)- n_k\cdot(\mathcal{E}(\mu_\infty)+2\gamma)-\log{(2Nn_k)}.
\end{align*}
Therefore, 
\begin{equation}\label{3.6}
	\frac{1}{n_k} H_{\nu_{k}}\left((\alpha_{l})_{0}^{n_{k}-1}\right)\ge P(T, \mathcal{E}; \mathcal{U})-  \mathcal{E}(\mu_\infty)-3\gamma-\frac{\log{(2Nn_k)}}{n_k}.
\end{equation}

Fix $m,n_{k}\in \mathbb{N}$ with $n_{k}>l$ and $1\leq m\leq n_{k}-1.$ Let $a(j)=[\frac{n_{k}-j}{m}], j = 0, 1, \dots,m-1$. Then  
\begin{equation}\nonumber
\begin{aligned}
\mathop{\bigvee}\limits_ {i=0}^{n_{k}-1}T^{-i}\alpha_{l}= \mathop{\bigvee}\limits_ {r=0}^{a(j)-1}T^{-(mr+j)}(\alpha_{l})_{0}^{m-1}\vee \mathop{\bigvee}\limits_ {t\in S_{j}}T^{-t}\alpha_{l},
\end{aligned}
\end{equation}
where $S_{j}=\{0, 1, \dots,j-1\}\cup\{j+ma(j),\dots,n_{k}-1\}.$ Since  $|S_{j}|\leq 2m,$ by the sub-additivity of $H_{\nu_k}(\cdot)$,
\begin{align*}
H_{\nu_{k}}\left((\alpha_{l})_{0}^{n_{k}-1}\right)
&\leq \mathop{\sum}\limits_ {r=0}\limits ^{a(j)-1}H_{\nu_{k}}\left(T^{-(mr+j)}(\alpha_{l})_{0}^{m-1}\right)+H_{\nu_{k}}\left(\mathop{\bigvee}\limits_ {l\in S_{j}}T^{-l}\alpha_{l}\right)\\
&\leq \mathop{\sum}\limits_ {r=0}\limits ^{a(j)-1}H_{T^{(mr+j)}_{*}\nu_{k}}\left((\alpha_{l})_{0}^{m-1}\right)+2m\log d,
\end{align*}
 where $d=|\mathcal{U}|$ represents the cardinality of $\mathcal{U}.$
Summing $j$   from $0$ to $m-1$, we obtain 
\begin{align*} 
H_{\nu_{k}}\left((\alpha_{l})_{0}^{n_{k}-1}\right)
&\leq \frac{1}{m}\mathop{\sum}\limits_ {j=0}\limits ^{m-1}\mathop{\sum}\limits_ {r=0}\limits ^{a(j)-1}H_{T_{*}^{(mr+j)}\nu_{k}}((\alpha_{l})_{0}^{m-1})+2m\log d\\
&\leq \frac{1}{m}\mathop{\sum}\limits_ {j=0}\limits ^{n_{k}-1}H_{T_{*}^{j}\nu_{k}}((\alpha_{l})_{0}^{m-1})+2m\log d.
\end{align*}
Due to the concavity of $H_{\{\cdot{}\}}((\alpha_{l})_{0}^{m-1})$ over $\mathcal{M}(X),$
\begin{equation*} 
\begin{aligned}
\frac{1}{n_{k}}\mathop{\sum}\limits_ {j=0}\limits ^{n_{k}-1}H_{T_{*}^{j}\nu_{k}}((\alpha_{l})_{0}^{m-1})\leq H_{\mu_{k}}((\alpha_{l})_{0}^{m-1}).
\end{aligned}
\end{equation*}
We conclude that 
\begin{equation}\label{3.7}
\begin{aligned}
\frac{1}{n_k}H_{\nu_{k}}\left((\alpha_{l})_{0}^{n_{k}-1}\right)\leq \frac{1}{m}H_{\mu_{k}}\left((\alpha_{l})_{0}^{m-1}\right)+\frac{2m\log d}{n_k}.
\end{aligned}
\end{equation}
From \eqref{3.6} and \eqref{3.7},  we obtain
\[P(T, \mathcal{E}; \mathcal{U})-  \mathcal{E}(\mu_\infty)-3\gamma-\frac{\log{(2Nn_k)}}{n_k}\le \frac{1}{m}H_{\mu_{k}}\left((\alpha_{l})_{0}^{m-1}\right)+\frac{2m\log d}{n_k}. \]
Since $\alpha_l$ is clopen, taking   $k\to\infty$,  one has 
\[P(T, \mathcal{E}; \mathcal{U})-  \mathcal{E}(\mu_\infty)-3\gamma   \le \frac{1}{m}H_{\mu_{\infty}}\left((\alpha_{l})_{0}^{m-1}\right). \]
The result follows by letting $m\to\infty$.  

{\bf Step 2.} In the general case, if $(X,T)$ is not zero-dimensional. Recall that $(X,T)$  always has a  zero-dimensional extension (see \cite{F.B}). Let   $(Y,S)$  be  an invertible zero-dimensional  TDS  and  $\pi:Y\rightarrow X$  be the  factor map    from $(Y,S)$ to $(X,T).$  

 Since    $\mathcal{E}\colon \mathcal{M}(X)\to\mathbb{R}$ is an energy   and $\mathcal{U}\in \mathcal{C}_{X}^{o}$, one has  $\mathcal{E}\circ \pi_*: \mathcal{M}(Y)\to \mathbb{R}$ is an energy   and  $\pi ^{-1}\mathcal{U}\in \mathcal{C}_{Y}^{0}$. Then by the previous case, 
 	$$ P(S,\mathcal{E}\circ \pi_*;\pi^{-1}\mathcal{U})\leq \mathop{\sup}\limits _{\nu \in \mathcal{M}_S(Y)}\left\{h_{\nu}(S,\pi^{-1}\mathcal{U})+\mathcal{E}\circ \pi_*(\nu)\right\}.$$
 	Note that $\pi_*(\mathcal{M}_S(Y))\subset \mathcal{M}(X,T).$ By  Lemma \ref{lem 2.5} and \ref{lem 2.4}, 
 	\begin{align*}
 			P(T,\mathcal{E}; \mathcal{U})&= P(S,\mathcal{E}\circ \pi_*;\pi^{-1}\mathcal{U})\\
 			&\leq \mathop{\sup}\limits _{\nu \in \mathcal{M}_S(Y)}\left\{h_{\nu}(S,\pi^{-1}\mathcal{U})+\mathcal{E}\circ \pi_*(\nu)\right\}\\
 			&=  \mathop{\sup}\limits _{\nu \in \mathcal{M}_S(Y)}\left\{h_{\pi_*\nu}(T, \mathcal{U})+\mathcal{E}(\pi_*\nu)\right\}\\
 			&\le \mathop{\sup}\limits _{\mu \in \mathcal{M}(X,T)}\left\{h_{\mu}(T, \mathcal{U})+\mathcal{E}(\mu)\right\}.
 	\end{align*}
 The proof is complete. 
\end{proof}

%
%
%

\section{Relationship between nonlinear topological pressure and nonlinear local pressure}\label{S3}
In this section, we give several equivalent definitions of nonlinear topological pressure.

Let $(X,T)$ be  a TDS .  For $n\in \mathbb{N},\epsilon>0$ and $x\in X$, denote  
$$B_n(x,\varepsilon)=\{y\in X:d(T^{i}x,T^{i}y)<\epsilon,\forall 0\le i\le n-1\}.$$
  A set $E\subset X$ is said to be an  {\em $(n,\epsilon)\text{-}$separated } set of $X$ with respect to $T$,  if for any distinct $x,y\in E$, $y\notin B_n(x,\varepsilon)$.  A set $D\subset X$ is said to be an {\em $(n,\epsilon)\text{-}$spanning}  set of $X$ with respect to $T$, if for any $x\in X$, there exists $y\in D$ such that $y\in B_n(x,\varepsilon)$.

\begin{definition} \label{DP} Let $\mathcal{E}\colon \mathcal{M}(X)\to \mathbb{R}$ be an energy. For  $n\geq 1$ and $\epsilon>0$, let 
	$$P_{n}(T,\mathcal{E},\varepsilon)=\sup \left\{\mathop{\sum}\limits _{x\in E} e^{n\mathcal{E}(\triangle_{x}^{n})}:E \;is\; an\;(n,\varepsilon)\text{-} separated\;  set \;of\; X \right\}.$$ Then the \emph{upper and lower nonlinear topological pressures} are defined respectively by
			\[P(T,\mathcal{E})=\lim\limits _{\epsilon\rightarrow0}\limsup \limits _{n\rightarrow \infty}\frac{1}{n}\log P_{n}(T,\mathcal{E},\varepsilon),\]
			\[\underline{P}(T,\mathcal{E})=\lim\limits _{\epsilon\rightarrow0}\liminf \limits _{n\rightarrow \infty}\frac{1}{n}\log P_{n}(T,\mathcal{E},\varepsilon).\]	 
\end{definition}

\begin{definition}\label{DQ} Let $\mathcal{E}\colon \mathcal{M}(X)\to \mathbb{R}$ be an energy. For  $n\geq 1$ and $\epsilon>0$, let 
	$$Q_{n}(T,\mathcal{E},\varepsilon)=\inf \left\{\mathop{\sum}\limits _{x\in F} e^{n\mathcal{E}(\triangle_{x}^{n})}:F \;is\; an\;(n,\varepsilon)\text{-} spanning\;  set \;of\; X \right\}.$$ Then we define
	\[Q(T,\mathcal{E})=\lim\limits _{\epsilon\rightarrow0}\limsup \limits _{n\rightarrow \infty}\frac{1}{n}\log Q_{n}(T,\mathcal{E},\varepsilon),\]
	\[\underline{Q}(T,\mathcal{E})=\lim\limits _{\epsilon\rightarrow0}\liminf \limits _{n\rightarrow \infty}\frac{1}{n}\log Q_{n}(T,\mathcal{E},\varepsilon).\]	 
\end{definition}

We also give several distinct definitions of nonlinear local pressure.

\begin{definition}\label{Dp12}
Let $\mathcal{E}\colon \mathcal{M}(X)\to \mathbb{R}$ be an energy and $\mathcal{U}\in \mathcal{C}_X^o$. For $n\ge1,$ denote
	$$p^1_{n}(T,\mathcal{E};\mathcal{U})=\inf \left\{\mathop{\sum}\limits _{V\in \mathcal{V}}\mathop{\sup} \limits _{x\in V}e^{n\mathcal{E}(\triangle_{x}^{n})}:\mathcal{V}\in \mathcal{C}_{X} \; and \; \mathcal{V}\succeq\mathcal{U}_{0}^{n-1}\right\},$$
	$$p^2_{n}(T,\mathcal{E};\mathcal{U})=\inf \left\{\mathop{\sum}\limits _{V\in \mathcal{V}}\mathop{\inf} \limits _{x\in V}e^{n\mathcal{E}(\triangle_{x}^{n})}:\mathcal{V}\in \mathcal{C}_{X} \; and \; \mathcal{V}\succeq\mathcal{U}_{0}^{n-1}\right\}.$$	
	Then we define
	$${P}_1(T,\mathcal{E};\mathcal{U})={P}(T,\mathcal{E};\mathcal{U})=\limsup \limits _{n\rightarrow \infty}\frac{1}{n}\log p^1_{n}(T,\mathcal{E};\mathcal{U}),$$
	$$\underline{P}_1(T,\mathcal{E};\mathcal{U})=\underline{P}(T,\mathcal{E};\mathcal{U})=\liminf \limits _{n\rightarrow \infty}\frac{1}{n}\log p^1_{n}(T,\mathcal{E};\mathcal{U}),$$
	and 
	$${P}_2(T,\mathcal{E};\mathcal{U})=\limsup \limits _{n\rightarrow \infty}\frac{1}{n}\log p^2_{n}(T,\mathcal{E};\mathcal{U}),$$
	$$\underline{P}_2(T,\mathcal{E};\mathcal{U})=\liminf \limits _{n\rightarrow \infty}\frac{1}{n}\log p^2_{n}(T,\mathcal{E};\mathcal{U}).$$
\end{definition}

\begin{definition}\label{Dp34}
	Let $\mathcal{E}\colon \mathcal{M}(X)\to \mathbb{R}$ be an energy and $\mathcal{U}\in \mathcal{C}_X^o$. For $n\ge1,$ denote
	$$p^3_{n}(T,\mathcal{E};\mathcal{U})=\inf\left\{\mathop{\sum}\limits _{B\in \mathcal{\beta}}\mathop{\sup} \limits _{x\in B}e^{n\mathcal{E}(\triangle_{x}^{n})}: \beta\text{~is a finite subcover of~}\mathcal{U}_{0}^{n-1}\right\},$$
	$$p^4_{n}(T,\mathcal{E};\mathcal{U})=\inf \left\{\mathop{\sum}\limits _{B\in \mathcal{\beta}}\mathop{\inf} \limits _{x\in B}e^{n\mathcal{E}(\triangle_{x}^{n})}: \beta\text{~is a finite subcover of~}\mathcal{U}_{0}^{n-1}\right\}.$$
	Then we define
	$${P}_3(T,\mathcal{E};\mathcal{U}) =\limsup \limits _{n\rightarrow \infty}\frac{1}{n}\log p^3_{n}(T,\mathcal{E};\mathcal{U}),$$
	$$\underline{P}_3(T,\mathcal{E};\mathcal{U})=\liminf \limits _{n\rightarrow \infty}\frac{1}{n}\log p^3_{n}(T,\mathcal{E};\mathcal{U}),$$
	and 
	$${P}_4(T,\mathcal{E};\mathcal{U})=\limsup \limits _{n\rightarrow \infty}\frac{1}{n}\log p^4_{n}(T,\mathcal{E};\mathcal{U}),$$
	$$\underline{P}_4(T,\mathcal{E};\mathcal{U})=\liminf \limits _{n\rightarrow \infty}\frac{1}{n}\log p^4_{n}(T,\mathcal{E};\mathcal{U}).$$
\end{definition}
\begin{lemma}\label{lem4.5}
  Let $(X,T)$ be a {\rm TDS},  $\mathcal{E}\colon \mathcal{M}(X)\to \mathbb{R}$ be an energy and $\mathcal{U}\in \mathcal{C}_X^o$.  Then for any $n\ge1$, 
  \[p^2_{n}(T,\mathcal{E};\mathcal{U})\le p^4_{n}(T,\mathcal{E};\mathcal{U})\le p^3_{n}(T,\mathcal{E};\mathcal{U}),\]
    \[p^2_{n}(T,\mathcal{E};\mathcal{U})\le p^1_{n}(T,\mathcal{E};\mathcal{U})\le p^3_{n}(T,\mathcal{E};\mathcal{U}).\]
\end{lemma}
\begin{proof}
	Note that  if $\beta$ is a finite subcover of $\mathcal{U}_{0}^{n-1}$, then $\beta\in \mathcal{C}_{X}$ and $ \beta\succeq\mathcal{U}_{0}^{n-1}.$ This immediately implies $p^2_{n}(T,\mathcal{E};\mathcal{U})\le p^4_{n}(T,\mathcal{E};\mathcal{U})$ and  $p^1_{n}(T,\mathcal{E};\mathcal{U})\le p^3_{n}(T,\mathcal{E};\mathcal{U})$. The remaining inequalities   $p^4_{n}(T,\mathcal{E};\mathcal{U})\le p^3_{n}(T,\mathcal{E};\mathcal{U})$ and $p^2_{n}(T,\mathcal{E};\mathcal{U})\le p^1_{n}(T,\mathcal{E};\mathcal{U})$ follow directly from comparing the infimum/supremum operations in their definitions.
\end{proof}

\begin{lemma}\label{lem4.6}
	Let $(X,T)$ be a \textit{TDS}, $\mathcal{E}\colon \mathcal{M}(X) \to \mathbb{R}$ be an energy. For any $\varepsilon > 0$,  let $\mathcal{U} \in \mathcal{C}_X^o$ with $\mathrm{diam}(\mathcal{U}) \le \varepsilon$ and denote
	\[
	\tau_\varepsilon = \sup \bigl\{ |\mathcal{E}(\mu) - \mathcal{E}(\nu)| : \mu,\nu \in \mathcal{M}(X), \, W(\mu,\nu) \le \varepsilon \bigr\}.
	\]
	
	\begin{enumerate}
		\item If $\delta > 0$ is the Lebesgue number of $\mathcal{U}$, then for any $n\ge 1$,
		\[
		e^{-n\tau_\varepsilon} \cdot p^3_{n}(T,\mathcal{E};\mathcal{U}) \le Q_{n}(T,\mathcal{E},\delta/2);
		\]
		\item For any $n \ge 1$, one has 
		\[
		Q_{n}(T,\mathcal{E},\varepsilon) \le P_{n}(T,\mathcal{E},\varepsilon) \le e^{n\tau_\varepsilon} \cdot p^2_{n}(T,\mathcal{E};\mathcal{U}).
		\]
	\end{enumerate}
\end{lemma}
\begin{proof}
	\textup{(\romannumeral1)} 
	Let $\mathcal{U} \in \mathcal{C}_X^o$ with $\mathrm{diam}(\mathcal{U}) \le \varepsilon$. For any $B \in \mathcal{U}_{0}^{n-1}$ and any $x, y \in B$, we estimate the Wasserstein distance between empirical measures:
	\begin{align*}
		W(\triangle_{x}^{n}, \triangle_{y}^{n}) 
		&= \sup \left\{ \triangle_{x}^{n}(g) - \triangle_{y}^{n}(g) : ~g \colon X \to \mathbb{R} \text{ is } 1\text{-Lipschitz} \right\} \\
		&= \sup \left\{ \sum_{i=0}^{n-1} \frac{g(T^{i}x) - g(T^{i}y)}{n} : ~g \colon X \to \mathbb{R} \text{ is } 1\text{-Lipschitz} \right\} \\
		&\leq \sum_{i=0}^{n-1} \frac{d(T^{i}x, T^{i}y)}{n} \\
		&\le \varepsilon.
	\end{align*}
	Therefore, by the definition of $p^i_{n}(T,\mathcal{E};\mathcal{U})$ and $\tau_\varepsilon$,
	\begin{equation}\label{4.1}
		p^1_{n}(T,\mathcal{E};\mathcal{U}) \le e^{n\tau_\varepsilon} \cdot p^2_{n}(T,\mathcal{E};\mathcal{U}),
	\end{equation}
	and 
	\begin{equation}\label{4.2}
		p^3_{n}(T,\mathcal{E};\mathcal{U}) \le e^{n\tau_\varepsilon} \cdot p^4_{n}(T,\mathcal{E};\mathcal{U}).
	\end{equation}
	
	Now let $F$ be an $(n, \delta/2)$-spanning set of $X$. Then
	\[
	X = \bigcup_{x \in F} B_n(x, \delta/2).
	\]
	Since $\delta$ is the Lebesgue number of $\mathcal{U}$,  each $B(T^{i}x, \delta/2)$ is contained in some element of $\mathcal{U}$. Thus 
$B_n(x, \delta/2) = \bigcap_{i=0}^{n-1} T^{-i} B(T^{i}x, \delta/2)$
	is contained in some element of $\mathcal{U}_{0}^{n-1}$. It follows that 
\begin{equation}\label{4.2.1}
		p^4_{n}(T,\mathcal{E};\mathcal{U}) \le \sum_{x \in F} e^{n\mathcal{E}(\triangle_{x}^{n})}.
\end{equation} 
	Combining with \eqref{4.2}, we obtain
	\[
	e^{-n\tau_\varepsilon} \cdot p^3_{n}(T,\mathcal{E};\mathcal{U}) \le Q_{n}(T,\mathcal{E},\delta/2).
	\]
	
	\textup{(\romannumeral2)} 
	Let $E$ be an $(n,\varepsilon)$-separated set of $X$ with maximum cardinality. Then $E$ is also   $(n,\varepsilon)$-spanning, giving  
	\[
	Q_{n}(T,\mathcal{E},\varepsilon) \le P_{n}(T,\mathcal{E},\varepsilon).
	\]	
	For $\mathcal{U} \in \mathcal{C}_X^o$ with $\mathrm{diam}(\mathcal{U}) \le \varepsilon$, each element of $\mathcal{U}_{0}^{n-1}$ contains at most one point in $E$. Hence
	\[
	\sum_{x \in E} e^{n\mathcal{E}(\triangle_{x}^{n})} \le p^1_{n}(T,\mathcal{E};\mathcal{U}).
	\]
	Applying  \eqref{4.1}  completes the proof: 
	\[
	P_{n}(T,\mathcal{E},\varepsilon) \le e^{n\tau_\varepsilon} \cdot p^2_{n}(T,\mathcal{E};\mathcal{U}).
	\]
\end{proof}

\begin{theorem}\label{Thm4.7}
Let $(X,T)$ be a {\rm TDS},  $\mathcal{E}\colon \mathcal{M}(X)\to \mathbb{R}$ be an energy.  Then the following quantities all coincide with $P(T,\mathcal{E}).$
\begin{enumerate}
    \item  ${Q}(T,\mathcal{E});$
    \item $\lim\limits _{\varepsilon\rightarrow0}\sup\limits_{\mathcal{U}\in \mathcal{C}_{X}^{o}}\big \{ P_i(T,\mathcal{E};\mathcal{U}):{\rm diam} (\mathcal{U})\leq \varepsilon \big \}$,\quad $\forall~ i\in \{1,2,3,4\};$
     \item $\lim\limits _{k\rightarrow\infty}P_i(T,\mathcal{E};\mathcal{U}_k)$,  \quad$\forall$ $\{\mathcal{U}_k\}\subset \mathcal{C}_{X}^{o}$ with ${\rm diam} (\mathcal{U}_k)\to0$, $\forall i\in \{1,2,3,4\};$
    \item $\sup\limits_{\mathcal{U}\in \mathcal{C}_{X}^{o}}P_i(T,\mathcal{E};\mathcal{U})$,\quad $\forall~ i\in \{1,2,4\}.$ 
\end{enumerate}
\end{theorem}

\begin{proof}
	$\text{(\romannumeral1)}$ \& $\text{(\romannumeral2)}$ Let $\varepsilon>0$ and $\mathcal{U} \in \mathcal{C}_X^o$ with $\mathrm{diam}(\mathcal{U}) \le \varepsilon$.  Let $\delta>0$ be the  Lebesgue number of $\mathcal{U}$. Then by    Lemma \ref{lem4.6}$\text{(\romannumeral1)}$, 
		\[	e^{-n\tau_\varepsilon} \cdot p^3_{n}(T,\mathcal{E};\mathcal{U}) \le Q_{n}(T,\mathcal{E},\delta/2), \quad \forall n \ge 1.	\]
Consequently,  for any  $\mathcal{U} \in \mathcal{C}_X^o$ with $\mathrm{diam}(\mathcal{U}) \le \varepsilon$,
\[-\tau_\varepsilon+ P_3(T,\mathcal{E};\mathcal{U}) \le \limsup\limits_{n\to\infty}\frac{1}{n}\log Q_{n}(T,\mathcal{E},\delta/2)\le Q(T,\mathcal{E}).\]
 Since $\tau_\varepsilon\to 0$ as $\varepsilon\to 0$,  we have
 \begin{equation}\label{4.3}
 \lim\limits _{\varepsilon\rightarrow0}\sup\limits_{\mathcal{U}\in \mathcal{C}_{X}^{o}}\big \{ P_3(T,\mathcal{E};\mathcal{U}):{\rm diam} (\mathcal{U})\leq \varepsilon \big \} \le Q(T,\mathcal{E}).
 \end{equation}
Similarly, by    Lemma \ref{lem4.6}$\text{(\romannumeral2)}$,  we may also obtain
\begin{equation}\label{4.4}
  Q(T,\mathcal{E})\le P(T,\mathcal{E})\le	 \lim\limits _{\varepsilon\rightarrow0}\sup\limits_{\mathcal{U}\in \mathcal{C}_{X}^{o}}\big \{ P_2(T,\mathcal{E};\mathcal{U}):{\rm diam} (\mathcal{U})\leq \varepsilon \big \}.
\end{equation}
By Lemma \ref{lem4.5},
\begin{equation}\label{4.5}
	P_2(T,\mathcal{E};\mathcal{U})\le P_4(T,\mathcal{E};\mathcal{U})\le P_3(T,\mathcal{E};\mathcal{U}),
\end{equation}
and
\begin{equation}\label{4.6}
	P_2(T,\mathcal{E};\mathcal{U})\le P_1(T,\mathcal{E};\mathcal{U})\le P_3(T,\mathcal{E};\mathcal{U}).
\end{equation}
Combining \eqref{4.3}-\eqref{4.6}, we conclude   
 \[P(T,\mathcal{E})= Q(T,\mathcal{E})=\lim\limits _{\varepsilon\rightarrow0}\sup\limits_{\mathcal{U}\in \mathcal{C}_{X}^{o}}\big \{ P_i(T,\mathcal{E};\mathcal{U}):{\rm diam} (\mathcal{U})\leq \varepsilon \big \},\quad \forall~ i\in \{1,2,3,4\}.\]
 The same argument proves \textup{(\romannumeral3)}.  
 
 \textup{(\romannumeral4)}  
 For any $\mathcal{U}\in \mathcal{C}_{X}^{o},$ suppose that  $\delta$ is the Lebesgue number of $\mathcal{U}$.  Let $F$ be an $(n, \delta/2)$-spanning set of $X$. Then by \eqref{4.2.1} and Lemma \ref{lem4.5},
 \[p^2_{n}(T,\mathcal{E};\mathcal{U}) \le p^4_{n}(T,\mathcal{E};\mathcal{U}) \le Q_{n}(T,\mathcal{E},\delta/2)\le P(T,\mathcal{E}).\]
 Hence,  
 \[\sup\limits_{\mathcal{U}\in \mathcal{C}_{X}^{o}}P_2(T,\mathcal{E};\mathcal{U})\le \sup\limits_{\mathcal{U}\in \mathcal{C}_{X}^{o}}P_4(T,\mathcal{E};\mathcal{U})\le P(T,\mathcal{E}). \]
 From part \textup{(\romannumeral2)},  for $i\in\{1,2\}$,
 $$P(T,\mathcal{E})=\lim\limits _{\varepsilon\rightarrow0}\sup\limits_{\mathcal{U}\in \mathcal{C}_{X}^{o}}\big \{ P_i(T,\mathcal{E};\mathcal{U}):{\rm diam} (\mathcal{U})\leq \varepsilon \big \}\le \sup\limits_{\mathcal{U}\in \mathcal{C}_{X}^{o}}P_i(T,\mathcal{E};\mathcal{U}).$$
 Therefore, 
 \[ \sup\limits_{\mathcal{U}\in \mathcal{C}_{X}^{o}}P_2(T,\mathcal{E};\mathcal{U})= \sup\limits_{\mathcal{U}\in \mathcal{C}_{X}^{o}}P_4(T,\mathcal{E};\mathcal{U})= P(T,\mathcal{E})\le \sup\limits_{\mathcal{U}\in \mathcal{C}_{X}^{o}}P_1(T,\mathcal{E};\mathcal{U}). \]
 It's left to prove 
 \[ P(T,\mathcal{E})\ge \sup\limits_{\mathcal{U}\in \mathcal{C}_{X}^{o}}P_1(T,\mathcal{E};\mathcal{U}). \]
 
  Fix an open cover $\mathcal{U} \in \mathcal{C}_X^o$ with Lebesgue number $\delta$, and let $n \geq 1$ be an arbitrary integer. Choose $x_1\in X$ such that $e^{n  \mathcal{E}(\triangle_{x_1}^{n})}=\sup_{x\in X}e^{n  \mathcal{E}(\triangle_{x}^{n})}$ and select $U_{i_1}\in \mathcal{U}_0^{n-1}$ satisfying $B_n(x_1,\delta/2)\subset U_{i_1}$. Denote $V_1=U_{i_1}$. Next, pick $x_2\in X\setminus V_1$ such that 
   $$e^{n  \mathcal{E}(\triangle_{x_2}^{n})}=\sup_{x\in X\setminus V_1}e^{n  \mathcal{E}(\triangle_{x}^{n})},$$  
   and take $U_{i_2}\in \mathcal{U}_0^{n-1}$ with $B_n(x_2,\delta/2)\subset U_{i_2}$. Let $V_2=U_{i_2}\setminus U_{i_1}$. Since $X$ is compact,   we can repeat this procedure inductively to obtain a finite set $E=\{x_1,\dots,x_m\}$, a collection  $\{U_{i_k}\}_{k=1}^m\subset \mathcal{U}_0^{n-1}$ and a disjoint cover $\mathcal{V}=\{V_1,\dots, V_m\}$ such that  for every $1\le k\le m,$
   \begin{enumerate} 
   	\item $B_n(x_k,\delta/2)\subset U_{i_k}$, $x_k\in V_k=U_{i_k}\setminus \cup_{j=1}^{k-1}U_{i_j}$,  and  $X=\cup_{k=1}^mV_k;$
   	\item $e^{n  \mathcal{E}(\triangle_{x_k}^{n})}=\sup\limits_{x\in X\setminus \cup_{j=1}^{k-1}V_j}e^{n  \mathcal{E}(\triangle_{x}^{n})}.$
   \end{enumerate}
   By construction, $E$ is an $(n,\delta/2)$-separated set of $X$, $\mathcal{V}\in\mathcal{C}_X$ and $\mathcal{V}\succeq\mathcal{U}_{0}^{n-1}$. Moreover, 
   \begin{align*}
   	\mathop{\sum}\limits _{V\in \mathcal{V}}\mathop{\sup} \limits _{x\in V}e^{n\mathcal{E}(\triangle_{x}^{n})}=\sum_{k=1}^{m}e^{n\mathcal{E}(\triangle_{x_k}^{n})}=\sum_{x\in E}e^{n\mathcal{E}(\triangle_{x}^{n})}.
   \end{align*}
   It follows that 
   \[p^1_{n}(T,\mathcal{E};\mathcal{U}) \le P_{n}(T,\mathcal{E},\delta/2)\le P(T,\mathcal{E}).\]
   Therefore,
   \[  P_1(T,\mathcal{E};\mathcal{U})\le P(T,\mathcal{E}),\quad \forall \mathcal{U}\in \mathcal{C}_{X}^{o}, \]
    which completes the proof. 
\end{proof}

\begin{remark}
	There exist cases where $\sup_{\mathcal{U}\in \mathcal{C}_{X}^{o}}P_3(T,\mathcal{E};\mathcal{U})>P(T,\mathcal{E}).$ As an example, consider the space $X=\{p\}\sqcup\Sigma_2$, where $\Sigma_2=\{0,1\}^\mathbb{Z}$ is the full two-sided shift space. Let $T:X\to X$ be defined by $T(p)=p$, and $T|_{\Sigma_2}$ acts as the left shift map.   Let $f:X\to \mathbb{R}$ be a function defined by $f(p)=10$ and $f|_{\Sigma_2}\equiv0$. The energy $\mathcal{E}$ is defined by $\mathcal{E}(\mu)=\int f d\mu$. Then 
	\begin{align*}
		 P(T,\mathcal{E})&=\mathop{\sup}\limits_{\mu \in \mathcal{M}^e(X,T)}\left\{h_{\mu}(T)+\int f d\mu\right\}\\
		                               &=\max\left\{h_{\mu_p}(T)+\int f d\mu_p, \mathop{\sup}\limits_{\mu \in \mathcal{M}^e(\Sigma_2,T)}\left\{h_{\mu}(T)+\int f d\mu\right\}\right\}\\
		                               &=10.
	\end{align*}
However, if we choose  $\mathcal{U}=\{\{p\}\cup [0], \{p\}\cup [1]\}$, where  $[0],[1]$ are cylinder sets in $\Sigma_2$,  then 
\[p^3_{n}(T,\mathcal{E};\mathcal{U})=\mathop{\sum}\limits _{B\in \mathcal{U}_{0}^{n-1}}\mathop{\sup} \limits _{x\in B}e^{n\mathcal{E}(\triangle_{x}^{n})}=2^n\cdot e^{n\cdot f(p)} .\]
As a consequence, $$P_3(T,\mathcal{E};\mathcal{U})=\log 2+10>P(T,\mathcal{E}).$$
\end{remark}

By replacing all instances of $\limsup$ with $\liminf$ in the proof of Theorem~\ref{Thm4.7}, we immediately obtain the following analogous result. As the argument requires no substantive changes beyond this substitution, we state the result without proof.
\begin{theorem}\label{Thm4.8}
	Let $(X,T)$ be a {\rm TDS},  $\mathcal{E}\colon \mathcal{M}(X)\to \mathbb{R}$ be an energy.  Then each of the following equals $\underline{P}(T,\mathcal{E}).$
	\begin{enumerate}
		\item  $\underline{Q}(T,\mathcal{E});$
		\item $\lim\limits _{\delta\rightarrow0}\sup\limits_{\mathcal{U}\in \mathcal{C}_{X}^{o}}\big \{ \underline{P}_i(T,\mathcal{E};\mathcal{U}):{\rm diam} (\mathcal{U})\leq \delta\big \}$,\quad $\forall~ i\in \{1,2,3,4\};$
		\item $\lim\limits _{k\rightarrow\infty}\underline{P}_i(T,\mathcal{E};\mathcal{U}_k)$,  \quad$\forall$ $\{\mathcal{U}_k\}\subset \mathcal{C}_{X}^{o}$ with ${\rm diam} (\mathcal{U}_k)\to0$, $\forall i\in \{1,2,3,4\};$
		\item $\sup\limits_{\mathcal{U}\in \mathcal{C}_{X}^{o}}\underline{P}_i(T,\mathcal{E};\mathcal{U})$,\quad $\forall~ i\in \{1,2,4\}.$ 
		
	\end{enumerate}
\end{theorem}

\bibliographystyle{amsplain}

\begin{thebibliography}{10}
	
	\bibitem{Barreira2}
	L.~Barreira and C.~Holanda, \emph{Higher-dimensional nonlinear thermodynamic
		formalism}, J. Stat. Phys. \textbf{187} (2022), no.~2, Paper No. 18, 28.
	
	\bibitem{Barreira1}
	\bysame, \emph{Nonlinear thermodynamic formalism for flows}, Dyn. Syst.
	\textbf{37} (2022), no.~4, 603--629.
	
	\bibitem{F.B}
	F.~Blanchard, E.~Glasner, and B.~Host, \emph{A variation on the variational
		principle and applications to entropy pairs}, Ergodic Theory Dynam. Systems
	\textbf{17} (1997), no.~1, 29--43.
	
	\bibitem{Buzzi}
	J.~Buzzi, B.~Kloeckner, and R.~Leplaideur, \emph{Nonlinear thermodynamical
		formalism}, Ann. H. Lebesgue \textbf{6} (2023), 1429--1477.
	
	\bibitem{CC}
	F.~Cai, \emph{Local weighted topological pressure}, J. Math. Phys. \textbf{65}
	(2024), no.~6, Paper No. 062701, 19.
	
	\bibitem{Danilenko}
	A.~Danilenko, \emph{Entropy theory from the orbital point of view}, Monatsh.
	Math. \textbf{134} (2001), no.~2, 121--141.
	
	\bibitem{DW}
	B.~Ding and T.~Wang, \emph{Some variational principles for nonlinear
		topological pressure}, Dyn. Syst. \textbf{40} (2025), no.~1, 35--55.
	
	\bibitem{DooleyZhang}
	A.~Dooley and G.~Zhang, \emph{Local entropy theory of a random dynamical
		system}, Mem. Amer. Math. Soc. \textbf{233} (2015), no.~1099, vi+106.
	
	\bibitem{Glasner05}
	E.~Glasner and B.~Weiss, \emph{On the interplay between measurable and
		topological dynamics}, Handbook of dynamical systems. {V}ol. 1{B}, Elsevier
	B. V., Amsterdam, 2006, pp.~597--648.
	
	\bibitem{Huangwentwo}
	W.~Huang, A.~Maass, P.~Romagnoli, and X.~Ye, \emph{Entropy pairs and a local
		{A}bramov formula for a measure theoretical entropy of open covers}, Ergodic
	Theory Dynam. Systems \textbf{24} (2004), no.~4, 1127--1153.
	
	\bibitem{Huangwenthree}
	W.~Huang, X.~Ye, and G.~Zhang, \emph{A local variational principle for
		conditional entropy}, Ergodic Theory Dynam. Systems \textbf{26} (2006),
	no.~1, 219--245.
	
	\bibitem{Huangwen}
	W.~Huang and Y.~Yi, \emph{A local variational principle of pressure and its
		applications to equilibrium states}, Israel J. Math. \textbf{161} (2007),
	29--74.
	
	\bibitem{Kucherenko}
	T.~Kucherenko, \emph{Nonlinear thermodynamic formalism through the lens of
		rotation theory}, Discrete Contin. Dyn. Syst. \textbf{44} (2024), no.~12,
	3760--3773.
	
	\bibitem{Leplaideur19}
	R.~Leplaideur and F.~Watbled, \emph{Generalized {C}urie-{W}eiss model and
		quadratic pressure in ergodic theory}, Bull. Soc. Math. France \textbf{147}
	(2019), no.~2, 197--219.
	
	\bibitem{LiangYan}
	B.~Liang and K.~Yan, \emph{Topological pressure for sub-additive potentials of
		amenable group actions}, J. Funct. Anal. \textbf{262} (2012), no.~2,
	584--601.
	
	\bibitem{MaChen}
	X.~Ma and E.~Chen, \emph{Variational principles for relative local pressure
		with subadditive potentials}, J. Math. Phys. \textbf{54} (2013), no.~3,
	032701, 25.
	
	\bibitem{MaChenZhang}
	X.~Ma, E.~Chen, and A.~Zhang, \emph{A relative local variational principle for
		topological pressure}, Sci. China Math. \textbf{53} (2010), no.~6,
	1491--1506.
	
	\bibitem{Romagnoli}
	P.~Romagnoli, \emph{A local variational principle for the topological entropy},
	Ergodic Theory Dynam. Systems \textbf{23} (2003), no.~5, 1601--1610.
	
	\bibitem{Romagnoli24}
	\bysame, \emph{A conditional variational principle for pressure of covers with
		respect to a partition}, Discrete Contin. Dyn. Syst. \textbf{44} (2024),
	no.~7, 2142--2168.
	
	\bibitem{R1}
	D.~Ruelle, \emph{Statistical mechanics on a compact set with {$Z\sp{v}$} action
		satisfying expansiveness and specification}, Trans. Amer. Math. Soc.
	\textbf{187} (1973), 237--251.
	
	\bibitem{SongLi}
	Z.~Song, Y.and~Li, \emph{A local conditional variational principle of pressures
		and local conditional equilibrium states}, J. Differential Equations
	\textbf{373} (2023), 476--525.
	
	\bibitem{Villani}
	C.~Villani, \emph{Optimal transport}, Grundlehren der mathematischen
	Wissenschaften, vol. 338, Springer-Verlag, Berlin, 2009, Old and new.
	
	\bibitem{W1}
	P.~Walters, \emph{A variational principle for the pressure of continuous
		transformations}, Amer. J. Math. \textbf{97} (1975), no.~4, 937--971.
	
	\bibitem{P.W}
	\bysame, \emph{An introduction to ergodic theory}, Graduate Texts in
	Mathematics, vol.~79, Springer-Verlag, New York-Berlin, 1982.
	
	\bibitem{WangWuXiao}
	M.~Wang, W.~Wu, and J.~Xiao, \emph{Local conditional topological pressure for
		amenable group actions}, J. Math. Phys. \textbf{66} (2025), no.~2, Paper No.
	022702, 20.
	
	\bibitem{WW}
	W.~Wu, \emph{Local pressure of subsets and measures}, J. Stat. Phys.
	\textbf{185} (2021), no.~2, Paper No. 9, 18.
	
	\bibitem{Yangchenzhou}
	J.~Yang, E.~Chen, and X.~Zhou, \emph{Variational principle for nonlinear
		weighted topological pressure}, J. Difference Equ. Appl. \textbf{31} (2025),
	no.~2, 188--208.
	
	\bibitem{Zhang09}
	G.~Zhang, \emph{Variational principles of pressure}, Discrete Contin. Dyn.
	Syst. \textbf{24} (2009), no.~4, 1409--1435.
	
\end{thebibliography}

\end{document}